\documentclass{amsart}
\usepackage[utf8]{inputenc}
\usepackage{amssymb,latexsym}
\usepackage{amsmath}
\usepackage{graphicx}
\usepackage{textcomp}

\usepackage{amsthm,amssymb,enumerate,graphicx, tikz}
\usepackage{amscd}
\usepackage{setspace}
\usepackage{comment}
\usepackage{hyperref}
\usepackage{cleveref}

\newtheorem{theorem}{Theorem}[section]

\newtheorem*{theorem*}{Theorem}

\theoremstyle{definition}
\newtheorem{definition}[theorem]{Definition}

\theoremstyle{remark}
\newtheorem{remark}[theorem]{Remark}

\newcommand{\F}{\mathcal{F}}

\newcommand{\conv}{\textrm{conv}}

\title{A complex analogue of the Goodman-Pollack-Wenger theorem}
\author{Daniel McGinnis}
\thanks{The author was supported by NSF grant DMS-1839918 (RTG)}

\begin{document}

\maketitle

\begin{abstract}
    A \textit{$k$-transversal} to family of sets in $\mathbb{R}^d$ is a $k$-dimensional affine subspace that intersects each set of the family.
    In 1957 Hadwiger provided a necessary and sufficient condition for a family of pairwise disjoint, planar convex sets to have a $1$-transversal. After a series of three papers among the authors Goodman, Pollack, and Wenger from 1988 to 1990, Hadwiger's Theorem was extended to necessary and sufficient conditions for $(d-1)$-transversals to finite families of convex sets in $\mathbb{R}^d$ with no disjointness condition on the family of sets. We prove an analogue of the Goodman-Pollack-Wenger theorem in the complex setting.
\end{abstract}

\section{Introduction}

The well-known Helly's theorem \cite{helly} states that if a finite family $\F$ of convex sets in $\mathbb{R}^d$ has the property that any choice of $d+1$ or less sets in $\F$ have a non-empty intersection, then there is a point in common to all the sets in $\F$ (see \cite{AmentaHelly2017} \cite{GoodmanDiscrete2017} for surveys on Helly's theorem and related results). A \textit{$k$-transversal} is a $k$-dimensional affine space that intersects each set of $\F$, so Helly's theorem provides a necessary and sufficient condition for $\F$ to have a $0$-transversal. In 1935, Vincensini was interested in the natural extension of Helly's theorem of finding necessary and sufficient conditions for a finite family of convex sets $\F$ in $\mathbb{R}^d$ to have a $k$-transversal for $k > 0$ \cite{vincensini1935figures}. In particular, Vincensini asked if there exists some constant $r=r(k,d)$ such that if every choice of $r$ or fewer sets in $\F$ has a $k$-transversal, then $\F$ has a $k$-transversal. However, Santal\'o provided examples showing that such a constant $r$ does not exist for any $k>0$ \cite{santalotheorem1940}. Many other related problems in geometric transversal theory have also been considered. For more information we refer the reader to the surveys \cite{GoodmanGeometric1993} \cite{GoodmanDiscrete2017}.

In 1957, Hadwiger made the first positive progress toward this extenson of Helly's theorem considered by Vincensini by proving the following theorem.

\begin{theorem}[Hadwiger \cite{HadwigerLines}]
    A finite family of pairwise disjoint convex sets in $\mathbb{R}^2$ has a $1$-transversal if and only if the sets in the family can be linearly ordered such that any three sets have a $1$-transversal consistent with the ordering.
\end{theorem}

Hadwiger's theorem has been generalized in different ways, eventually resulting in an encompassing result for $(d-1)$-transversals in $\mathbb{R}^d$.
The first significant result in this direction was made by Goodman and Pollack who showed that in $\mathbb{R}^d$, the linear ordering in Hadwiger's theorem can be replaced with the notion of an order type of points in $\mathbb{R}^{d-1}$ given the additional condition that the family is $(d-2)$-separable, which generalizes the disjiontness condition in Hadwiger's theorem (see \cite{GoodmanHadwiger1988} for the precise statement of the theorem and definitions of these notions). Soon after, Wenger showed that the disjointness condition of Hadwiger's theorem can be dropped \cite{WengerIntersecting1990}. Finally, Pollack and Wenger completed the picture by proving the necessary and sufficient conditions required to have a $(d-1)$-transversal in $\mathbb{R}^d$ with no additional separability conditions on the family of sets \cite{PollackNecessary1990}. We note that several extensions of this result, including colorful generalizations, have been studied for instance in \cite{Anderson1996Oriented} \cite{ArochaSpearoids2002} \cite{ArochaColorful2008} \cite{HolmsenColored2016} \cite{holmsen2022colorful}.

Despite the previous work on the existence of $(d-1)$-transversals, no necessary and sufficient conditions for the existence of $k$-transversals in $\mathbb{R}^d$ for $0<k<d-1$ have been proven or conjectured. Our attempts to find such a condition for $(d-2)$-transversals in $\mathbb{R}^d$ eventually led us to instead a necessary and sufficient condition for complex $(d-1)$-transversals in $\mathbb{C}^d$. The statement of this result appears in Section \ref{sec:main}.

\section{Hyperplane transversals revisited}\label{sec:hyperplanes}

Here we will describe the result of Pollack and Wenger on $(d-1)$-transversals in $\mathbb{R}^d$ as presented in \cite{GoodmanDiscrete2017}, then we will discuss an equivalent rephrasing of this theorem to put our main result in Section \ref{sec:main} into context.

Let $\F$ be a finite family of convex sets in $\mathbb{R}^d$ and let $P$ be a subset of points in $\mathbb{R}^k$ for some $k$. We say that $\F$ \textit{separates consistently} with $P$ if there exists a map $\phi:\F \rightarrow P$ such that for any two subfamilies $\F_1,\ \F_2\subset \F$, we have that
\[
\conv(\F_1) \cap \conv(\F_2) = \emptyset \implies \conv(\phi(\F_1))\cap \conv(\phi(\F_2)) = \emptyset.
\]
Here we mean $\conv(\F_i)$ to be $\conv(\cup_{F\in \F_i} F)$. Another way to think about this condition is that if the sets of $\F_1$ can be separated from the sets of $\F_2$ by a hyperplane in $\mathbb{R}^d$, then the sets of points $\phi(\F_1)$ and $\phi(\F_2)$ can be separated by a hyperplane in $\mathbb{R}^k$. We also note that $\F$ separates consistently with $P$ if and only if
\[
\conv(\F_1) \cap \conv(\F_2) = \emptyset \implies \conv(\phi(\F_1))\cap \conv(\phi(\F_2)) = \emptyset.
\]
whenever $|\F_1| + |\F_2| \leq k+2$. This is a consequence of the well-known Kirchberger's theorem \cite{KirchbergerTheorem}, which states that if $U$ and $V$ are finite point sets in $\mathbb{R}^k$ such that for every set of $k+2$ points $S\subset U\cup V$, we have that $\conv(S\cap U)\cap \conv(S\cap V) = \emptyset$, then $\conv(U) \cap \conv(V) = \emptyset$.

We now have the terminology to state the Goodman-Pollack-Wenger theorem.
\begin{theorem}[Goodman-Pollack-Wenger theorem \cite{PollackNecessary1990}]\label{thm:GPW}
    A finite family of convex sets $\F$ in $\mathbb{R}^d$ has a $(d-1)$-transversal if and only if $\F$ separates consistently with a set $P\subset \mathbb{R}^{d-1}$.
\end{theorem}

The condition in our main result of Section \ref{sec:main} is quite similar to the condition in Theorem \ref{thm:GPW}, and we will first provide a slight rephrasing of the definition for $\F$ to separate consistently with $P$ in order to make this similarity more apparent. By taking the contrapositive of the implication in the definition of separating consistently, we may equivalently say that $\F$ separates consistently with $P$ if there exists a map $\phi:\F \rightarrow P\subset \mathbb{R}^k$ such that
\[
\conv(\phi(\F_1))\cap \conv(\phi(\F_2)) \neq \emptyset \implies \conv(\F_1) \cap \conv(\F_2) \neq \emptyset.
\]
In other words, the existence of an affine dependence 
\[
\sum_{F\in \F_1\cup \F_2} a_F=0,\, \sum_{F\in \F_1\cup \F_2} a_F\phi(F)= 0
\]
where $a_F \geq 0$ for all $F\in \F_1$ (not all 0) and $a_F\leq 0$ for all $F\in \F_2$ implies the existence of points $p_F\in F$ and real numbers $r_F\geq 0$ such that 
\[
\sum_{F\in \F_1\cup \F_2} r_Fa_F=0,\, \sum_{F\in \F_1\cup \F_2} (r_Fa_F)p_F= 0
\]
is an affine dependence of the points $p_F$ and the numbers $r_Fa_F$ are not all $0$.

\section{Main result}\label{sec:main}

In this section, we provide a necessary and sufficient condition for a finite family $\F$ of convex sets in $\mathbb{C}^d$ to have a \textit{complex} $(d-1)$-transversal, where a complex $(d-1)$-transversal here is a complex $(d-1)$-dimensional affine subspace of $\mathbb{C}^d$ that intersects each set in $\F$.

First, following our discussion from Section \ref{sec:hyperplanes}, we make the following definition in order to articulate our main theorem, Theorem \ref{thm:main}.

\begin{definition}\label{def:main}
Let $\F$ be a finite family of convex sets in $\mathbb{C}^d$, and let $P\subset \mathbb{C}^k$. We say that $\F$ is \textit{dependency-consistent} with $P$ if there exists a map $\phi:\F \rightarrow P$ such that for every subfamily $\F'\subset \F$ and every affine dependence
\[
\sum_{F\in \F'}a_F=0,\, \sum_{F\in \F'}a_F\phi(F)=0
\]
for complex numbers $a_F$, there exist real numbers $r_F\geq 0$ and points $p_F\in F$ for $F\in \F'$ such that
\[
\sum_{F\in \F'}r_Fa_F=0,\, \sum_{F\in \F'}(r_Fa_F)p_F=0
\]
where not all of the values $r_Fa_F$ are 0.
\end{definition}

\begin{remark}\label{rmk}
For the purpose of Theorem \ref{thm:main}, we could add the additional restriction that $|\F'|\leq 2k+3$ in Definition \ref{def:main}, and the statement of Theorem \ref{thm:main} still holds. This is due to the following reasoning. By associating the points $(a_F\phi(F),a_F)$ with points in $\mathbb{R}^{2k+2}$, we have that the set of points $\{(a_F\phi(F),a_F)\}_{F\in \F'}$ contains $0\in \mathbb{R}^{2k+2}$ in its convex hull. Therefore, by Carath\'eodory's Theorem, there exist $m\leq 2k+3$ sets $F_1,\dots,F_m\in \F'$ and real numbers $s_i>0$ such that $\sum_{i=1}^m s_i(a_{F_i}\phi(F_i),a_{F_i}) = 0$. In other words, there is the complex affine dependence 
\[
\sum_{i=1}^m s_ia_{F_i}=0,\, \sum_{i=1}^m (s_ia_{F_i})\phi(F_i)=0
\]
among the points $\phi(F_1),\dots,\phi(F_m)$.

\end{remark}

In our proof of Theorem \ref{thm:main} will make use of the Borsuk-Ulam theorem below. We note that the Borsuk-Ulam theorem was also employed in the proof of the Goodman-Pollack-Wenger theorem in \cite{PollackNecessary1990}, and in fact our proof of Theorem \ref{thm:main} takes significant inspiration from this proof.
\begin{theorem}[Borsuk-Ulam theorem]\label{thm:borsuk}
If $n\geq m$ and $f:S^n \rightarrow \mathbb{R}^m$ is an odd, continuous map, i.e. $f(-x) = -f(x)$ for all $x\in S^n$, then $0\in \textrm{Im}(f)$.
\end{theorem}

We are now ready to state and prove the main result. The non-trivial direction, and the contents of the following proof, is the ``if'' part of the statement. The ``only if'' direction is straightforward since a complex $(d-1)$-transversal can be associated with $\mathbb{C}^{d-1}$, and $P$ can be constructed by choosing a point in the intersection of a set from $\F$ and the complex $(d-1)$-transversal for each set in $\F$.

\begin{theorem}[Main theorem]\label{thm:main}
    
A finite family of convex sets $\F$ in $\mathbb{C}^{d}$ has a complex $(d-1)$-transversal if and only if $\F$ is dependency-consistent with a set $P\subset \mathbb{C}^{d-1}$.
\end{theorem}

\begin{proof}
We associate $\mathbb{C}^d$ with the set of points 
\[
H=\{(z_1,\dots,z_{d+1})\in \mathbb{C}^{d+1} \mid z_{d+1}=1\}.
\]
Thus, we think of the convex sets from the family $\F$ as lying in $H$ as well.

We will now define a continuous odd map $f$ from $\{\mathbf{z}\in \mathbb{C}^{d+1} \mid ||\mathbf{z}||=1\}$, which can be identified with the $(2d+1)$-dimensional sphere, $S^{2d+1}$, to $\mathbb{R}^{2d}$. Such a map will have a zero by Theorem \ref{thm:borsuk}, and we will show that for the map that we define, this zero will correspond to a complex $(d-1)$-transversal of $\F$.

Let $\mathbf{x}\in \{\mathbf{z}\in \mathbb{C}^{d+1} \mid ||\mathbf{z}||=1\}$, and let $P_\mathbf{x}$ be the complex subspace spanned by $\mathbf{x}$. Additionally, for a set $F\in \F$ we denote the orthogonal projection of $F$ onto $P_\mathbf{x}$ by $F_\mathbf{x}$; note that $F_\mathbf{x}$ is a convex set.

For each set $F\in \F$, let $p_{\mathbf{x},F}$ be the complex number $c$ such that $c\mathbf{x}$ is the element of $F_\mathbf{x}$ that is closest in distance to $\mathbf{0}$, the origin in $\mathbb{C}^{d+1}$. Note that the convexity of $F_\mathbf{x}$ implies that $c$ is unique. Furthermore, for a fixed set $F$, the point $p_{\mathbf{x},F}$ varies continuously with $\mathbf{x}$. Let $\phi$ be the map witnessing the fact that $\F$ is dependency-consistent with $P$; we define $f(\mathbf{x})\in \mathbb{C}\times \mathbb{C}^{d-1}$, which can be identified with $\mathbb{R}^{2d}$, as follows

\[
f(\mathbf{x}) = \sum_{F\in \F}\left(p_{\mathbf{x},F},\, \overline{p_{\mathbf{x},F}}\phi(F)\right),
\]
where $\overline{p_{\mathbf{x},F}}$ denotes the complex conjugate of $p_{\mathbf{x},F}$. Because $p_{\mathbf{x},F}=-p_{-\mathbf{x},F}$ and $p_{\mathbf{x},F}$ varies continuously with $\mathbf{x}$, we have that $f$ is indeed a continuous odd map and thus has a zero, say $\mathbf{x}_0$, by Theorem \ref{thm:borsuk}.

It cannot be the case that $\mathbf{x}_0=(0,\dots,0,z_{d+1})$ for some $z_{d+1}\neq 0$, since in that case, each $F_x$ is the point $(0,\dots,0,1)$ and thus $p_{x_0,F} = 1/z_{d+1}$ for every $F \in \F$.   If $p_{x_0,F} = 1/z_{d+1}$ for each $F$, then clearly $\sum_{F \in \F} p_{x_0,F} \neq 0$. Thus, $\mathbf{x}_0\neq (0,\dots,0,z_{d+1})$, and in particular, this implies that the orthogonal complement of $\mathbf{x}_0$ intersects $H$ in a complex $(d-1)$-dimensional affine space, say $T$. If $T$ intersects each set in $\F$, then $\F$ has a complex $(d-1)$-transversal, and we are done. Therefore, we assume that $T$ does not intersect each set in $\F$, so we have that some of the values $p_{\mathbf{x}_0,F}$ must be nonzero. We show that this assumption leads to a contradiction.

Let $\F'\subset \F$ be the family of sets $F$ such that $p_{\mathbf{x}_0,F}\neq 0$. Since $f(\mathbf{x}_0) = 0$, we have the affine dependence
\[
\sum_{F\in \F'}p_{\mathbf{x}_0,F} = 0,\, \sum_{F\in \F'}\overline{p_{\mathbf{x}_0,F}}\phi(F)=0.
\]
Since $\F$ is dependency-consistent with $P$, we have that there exist points $p_F \in F$ for all $F\in \F'$ and real numbers $r_F\geq 0$ such that 
\[
\sum_{F\in \F'}r_Fp_{\mathbf{x}_0,F} = 0,\, \sum_{F\in \F'}\overline{r_Fp_{\mathbf{x}_0,F}}p_F=0.
\]
(Note that at this step, we could assume that $|\F'|\leq 2d+1$ by Remark \ref{rmk}). Now, consider the $\mathbb{C}$-linear map $\textrm{proj}_{\mathbf{x}_0}: \mathbb{C}^{d+1}\rightarrow P_{\mathbf{x}_0}$ given by orthogonal projection onto $P_{\mathbf{x}_0}$. By linearity, we have that 
\[
\textrm{proj}_{\mathbf{x}_0}\left(\sum_{F\in \F'}\overline{r_Fp_{\mathbf{x}_0,F}}p_F\right) = \sum_{F\in \F'}\overline{r_Fp_{\mathbf{x}_0,F}}\textrm{proj}_{\mathbf{x}_0}(p_F) = 0.
\]
However, this immediately leads to a contradiction. Indeed, recall that $p_{\mathbf{x}_0,F}\mathbf{x}_0$ is the point of $F_{\mathbf{x}_0}$ closest to $\mathbf{0}$ and $F_{\mathbf{x}_0}\ni \textrm{proj}_{\mathbf{x}_0}(p_F)$ is a convex set. This implies that the angle between $p_{\mathbf{x}_0,F}\mathbf{x}_0$ and $\textrm{proj}_{\mathbf{x}_0}(p_F)$ is less than 90 degrees (see Figure \ref{fig}). Therefore, writing $\textrm{proj}_{\mathbf{x}_0}(p_F) = c_F\mathbf{x}_0$, we have that the absolute difference in argument between the complex numbers $p_{\mathbf{x}_0,F}$ and $c_F$ is less than 90 degrees, which implies that $\textrm{Re}(\overline{p_{\mathbf{x}_0,F}}c_F) > 0$. This is in contradiction to the fact that $\sum_{F\in \F'}\overline{r_Fp_{\mathbf{x}_0,F}}\textrm{proj}_{\mathbf{x}_0}(p_F) = 0$, and hence completes the proof.

\end{proof}

\begin{figure}[h!]
    \centering
    \begin{tikzpicture}[scale = 1.2]
\filldraw [black] (0,0) circle (2pt) node[below]{$\mathbf{0}$};
        \draw [black] (50:3cm) circle [x radius=1.5cm, y radius=1 cm, rotate=240];
        \draw[->] [black] (0,0) -- (0.95*1.1,0.95*1.05);
        \filldraw [black] (1.1,1.05) circle (1pt) node[left]{$p_{\mathbf{x}_0,F}\mathbf{x}_0$};

        \draw[->] [black] (0,0) -- (0.97*2.5,0.97*1.7);
        \filldraw [black] (2.5,1.7) circle (1pt) node[right]{$\ \  \textrm{proj}_{\mathbf{x}_0}(p_F)$};

        \filldraw [black] (2,3) circle (0pt) node[below]{$F_{\mathbf{x}_0}$};
    \end{tikzpicture}
    \caption{Depiction of $F_{\mathbf{x}_0}$, $p_{\mathbf{x}_0,F}\mathbf{x}_0$, and $\textrm{proj}_{\mathbf{x}_0}(p_F)$ in the complex plane $P_{\mathbf{x}_0}$.}
    \label{fig}
\end{figure}
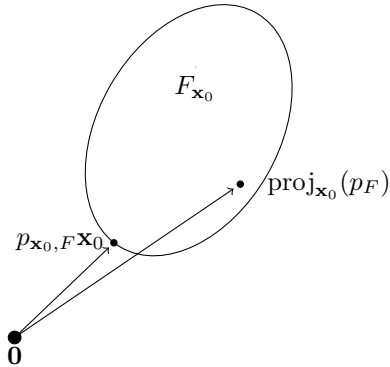

\section{Acknowledgements}
The author would like to thank Shira Zerbib for her helpful comments on a first draft of this paper, and the anonymous reviewers for their suggested improvements.

\newcommand{\etalchar}[1]{$^{#1}$}
\providecommand{\bysame}{\leavevmode\hbox to3em{\hrulefill}\thinspace}
\providecommand{\MR}{\relax\ifhmode\unskip\space\fi MR }
\providecommand{\MRhref}[2]{%
  \href{http://www.ams.org/mathscinet-getitem?mr=#1}{#2}
}
\providecommand{\href}[2]{#2}


\begin{thebibliography}{ABM{\etalchar{+}}02}

\bibitem[ABM{\etalchar{+}}02]{ArochaSpearoids2002}
J.~L. Arocha, J.~Bracho, L.~Montejano, D.~Oliveros, and R.~Strausz,
  \emph{Separoids, their categories and a {H}adwiger-type theorem for
  transversals}, Discrete Comput. Geom. \textbf{27} (2002), no.~3, 377--385.
  \MR{1921560}

\bibitem[ABM08]{ArochaColorful2008}
Jorge~L. Arocha, Javier Bracho, and Luis Montejano, \emph{A colorful theorem on
  transversal lines to plane convex sets}, Combinatorica \textbf{28} (2008),
  no.~4, 379--384. \MR{2452840}

\bibitem[ADLS17]{AmentaHelly2017}
Nina Amenta, Jes\'{u}s~A. De~Loera, and Pablo Sober\'{o}n, \emph{Helly's
  theorem: new variations and applications}, Algebraic and geometric methods in
  discrete mathematics, Contemp. Math., vol. 685, Amer. Math. Soc., Providence,
  RI, 2017, pp.~55--95. \MR{3625571}

\bibitem[AW96]{Anderson1996Oriented}
Laura Anderson and Rephael Wenger, \emph{Oriented matroids and hyperplane
  transversals}, Adv. Math. \textbf{119} (1996), no.~1, 117--125. \MR{1383885}

\bibitem[GOT18]{GoodmanDiscrete2017}
Jacob~E. Goodman, Joseph O'Rourke, and Csaba~D. T\'{o}th (eds.), \emph{Handbook
  of discrete and computational geometry}, Discrete Mathematics and its
  Applications (Boca Raton), CRC Press, Boca Raton, FL, 2018, Third edition of
  [ MR1730156]. \MR{3793131}

\bibitem[GP88]{GoodmanHadwiger1988}
Jacob~E. Goodman and Richard Pollack, \emph{Hadwiger's transversal theorem in
  higher dimensions}, J. Amer. Math. Soc. \textbf{1} (1988), no.~2, 301--309.
  \MR{928260}

\bibitem[GPW93]{GoodmanGeometric1993}
Jacob~E. Goodman, Richard Pollack, and Rephael Wenger, \emph{Geometric
  transversal theory}, New trends in discrete and computational geometry,
  Algorithms Combin., vol.~10, Springer, Berlin, 1993, pp.~163--198.
  \MR{1228043}

\bibitem[Had57]{HadwigerLines}
H.~Hadwiger, \emph{Ueber {E}ibereiche mit gemeinsamer {T}reffgeraden},
  Portugal. Math. \textbf{16} (1957), 23--29. \MR{99015}

\bibitem[Hel23]{helly}
E.~Helly, \emph{\"uber mengen konvexer k\"orper mit gemeinschaftlichen
  punkten}, Jahresber. Deutsch. Math. Verein. \textbf{32} (1923), 175--176.

\bibitem[Hol22]{holmsen2022colorful}
Andreas~F. Holmsen, \emph{A colorful goodman-pollack-wenger theorem}, 2022.

\bibitem[HRP16]{HolmsenColored2016}
Andreas~F. Holmsen and Edgardo Rold\'{a}n-Pensado, \emph{The colored {H}adwiger
  transversal theorem in {$\Bbb{R}^d$}}, Combinatorica \textbf{36} (2016),
  no.~4, 417--429. \MR{3537034}

\bibitem[Kir03]{KirchbergerTheorem}
Paul Kirchberger, \emph{\"{U}ber {T}chebychefsche {A}nn\"{a}herungsmethoden},
  Math. Ann. \textbf{57} (1903), no.~4, 509--540. \MR{1511222}

\bibitem[PW90]{PollackNecessary1990}
R.~Pollack and R.~Wenger, \emph{Necessary and sufficient conditions for
  hyperplane transversals}, Combinatorica \textbf{10} (1990), no.~3, 307--311.
  \MR{1092546}

\bibitem[San40]{santalotheorem1940}
L.~A. Santal\'{o}, \emph{A theorem on sets of parallelepipeds with parallel
  edges}, Publ. Inst. Mat. Univ. Nac. Litoral \textbf{2} (1940), 49--60.
  \MR{3726}

\bibitem[Vin35]{vincensini1935figures}
Paul Vincensini, \emph{Figures convexes et vari{\'e}t{\'e}s lin{\'e}aires de
  l’espace euclidien {\`a} n dimensions}, Bull. Sci. Math \textbf{59} (1935),
  163--174.

\bibitem[Wen90]{WengerIntersecting1990}
Rephael Wenger, \emph{A generalization of {H}adwiger's transversal theorem to
  intersecting sets}, Discrete Comput. Geom. \textbf{5} (1990), no.~4,
  383--388. \MR{1043720}

\end{thebibliography}
\end{document}